\documentclass{jasprealex}

\RequirePackage[OT1]{fontenc}
\RequirePackage{amsthm,amsmath,amsfonts,amssymb}
\RequirePackage[numbers]{natbib}
\RequirePackage[colorlinks,citecolor=blue,urlcolor=blue]{hyperref}
\RequirePackage{url}
\RequirePackage{epsfig}
\RequirePackage{algorithmic}

\numberwithin{equation}{section}
\theoremstyle{plain}

\newtheorem{Thm}{Theorem}[section]
\newtheorem{Cor}[Thm]{Corollary}
\newtheorem{Lem}[Thm]{Lemma}
\newtheorem{Prop}[Thm]{Proposition}

\theoremstyle{definition}

\newtheorem{Def}[Thm]{Definition}
\newtheorem{Xmp}[Thm]{Example}
\newtheorem{Rem}[Thm]{Remark}
\newtheorem{Alg}[Thm]{Algorithm}
\newtheorem{Pro}[Thm]{Problem}

\newcommand{\Prob}{{\bf P}}

\newcommand{\C}{{\mathbb C}}

\newcommand{\R}{{\mathbb R}}
\newcommand{\N}{{\mathbb N}}
\renewcommand{\S}{{\mathbb S}}

\renewcommand{\S}{{\Sigma}}

\newcommand{\card}{\mbox{\rm card~}}

\renewcommand{\det}{\mbox{\rm det\,}}
\renewcommand{\dim}{\mbox{\rm dim\,}}

\newcommand{\diag}{\mbox{\rm diag~}}

\newcommand{\Id}{\mbox{\rm Id~}}
\newcommand{\im}{\mbox{\rm Im~}}

\newcommand{\rad}{\mbox{\rm rad~}}
\newcommand{\rk}{\mbox{\rm rk~}}

\newcommand{\spann}{\mbox{\rm span}}

\newcommand{\cH}{{\mathcal H}}

\newcommand{\cM}{{\mathcal M}}
\newcommand{\cN}{{\mathcal N}}

\newcommand{\cP}{{\mathcal P}}

\newcommand{\cS}{{\mathcal S}}

\newcommand{\cX}{{\mathcal X}}

\begin{document}

\title{Generic Identification of Binary-Valued Hidden Markov Processes}

\author[A. Sch\"onhuth]{Alexander Sch\"onhuth\affil{1}\comma\corrauth}

\address{\affilnum{1}\ Centrum Wiskunde \& Informatica, Amsterdam,
  Netherlands}

\email{{\tt as@cwi.nl}}

\begin{abstract}
  The generic identification problem is to decide whether a stochastic
  process $(X_t)$ is a hidden Markov process and if yes to infer its
  parameters for all but a subset of parametrizations that form a
  lower-dimensional subvariety in parameter space. Partial answers so
  far available depend on extra assumptions on the processes, which
  are usually centered around stationarity. Here we present a general
  solution for binary-valued hidden Markov processes. Our approach is
  rooted in algebraic statistics hence it is geometric in nature. We
  find that the algebraic varieties associated with the probability
  distributions of binary-valued hidden Markov processes are zero sets
  of determinantal equations which draws a connection to well-studied
  objects from algebra. As a consequence, our solution allows for
  algorithmic implementation based on elementary (linear) algebraic
  routines.
\end{abstract}


\maketitle

\section{Introduction}
\label{sec.intro}

Hidden Markov processes (HMPs) have gained widespread interest in
statistics, predominantly due to their striking successes in
applications. Central theoretical concerns have revolved around the
fundamental problems of {\em identifiability}\/ and {\em complete
  identification}. Here and in the following, stochastic processes
$(X_t)$ take values in a finite set ({\em alphabet}\/) $\S$ where {\em
  binary-valued} refers to the case $|\S|=2$.

\begin{Pro}[Complete Identification]
  \label{pro.complete}
  Decide whether a stochastic process $(X_t)$ is a hidden Markov
  process. If it is, infer its parameters.
\end{Pro}

The problem was already raised in the late 50s. A representative list
of references is
\cite{Blackwell57,Gilbert59,Dharma63a,Dharma63b,Dharma65,Fox68,Petrie69,Erickson70}.
See also the more recent contributions
\cite{Baras92,Anderson99,Vidyasagar09,Finesso10} and the
(near-exhaustive) list of references in \cite{Ephraim02}. See also
\cite{Baras92,Baum66} for HMM parameter estimation from data and
\cite{Cappe} for a textbook on related practical issues. In terms of
practical arguments one can argue that it is reasonable to solve
Problem~\ref{pro.complete} for all but a null set of parametrizations
which also explains that the most recent contributions
\cite{Vidyasagar09,Finesso10} provide {\em generic} solutions. That
is, solutions apply for all, but a subset of parametrizations which
form a lower-dimensional subvariety in parameter space.\par

The above-mentioned treatments all raise extra assumptions on the
processes, usually centered around stationarity. The only exception is
Heller who provided a polyhedral cone-based characterization of
arbitrary, also non-stationary HMPs~\cite{Heller65}. This, however,
was exposed as a reformulation rather than a
solution~\cite{Anderson99} in the sense that it does not give rise to
an algorithmic solution of Problem \ref{pro.complete}. To date,
Problem \ref{pro.complete} has still not yet been fully resolved.\par
The fact that one can assign every probability distribution
$\Prob:\S^n\to[0,1]$ over finite-length strings to a HMP on $|\S|^n$
states (which is a well-known exercise, the hidden states of the HMP
form a de Bruijn graph over $\S^n$, together with the obvious
transition probabilities), introduces further complications when
aiming at algorithmic solutions.  We therefore turn our attention to
the following finite reformulation of Problem \ref{pro.complete}.

\begin{Pro}[Finite Identification]
\label{pro.finite}
Let $\Prob:\S^n\to[0,1]$ be a probability distribution over strings of
finite length $n$. Decide whether $\Prob$ is due to a HMP on $d$
hidden states. If it is, infer its parameters.
\end{Pro}

In the course of this paper, we provide a generic solution of Problem
\ref{pro.finite} for binary-valued alphabets in the case of
$d\le\frac{n+1}{2}$. Our solution is rooted in algebraic statistics
where we draw in particular from the concept of an algebraic
statistical model, as described in \cite{Pachter,Drton07}. See for
example \cite{Garcia05,Sullivant08} for discussions on Bayesian
networks, which, as latent variable models, are related to hidden
Markov models.  Since HMPs are uniquely determined by their
distributions over strings of length $2d-1$ \cite{Paz}, a solution of
Problem \ref{pro.finite} also gives rise to a solution of the original
Problem \ref{pro.complete}:

\begin{enumerate}
\item For each $n\in\N$ determine $d(n)$ as the minimal number of
  hidden states such that the answer in the 'Decision' part of
  Problem~\ref{pro.finite} is 'Yes'. In case that there is no
  $d\le\frac{n+1}{2}$ set $d(n):=\infty$.
\item If $d(n),n\in\N$ converges, output 'Yes' and infer corresponding
  parameters. If not, output 'No'.
\end{enumerate}

Since one can extend any probability distribution $\Prob:\S^n\to[0,1]$
that stems from a HMP, to a full, non-HMP stochastic process $(X_t)$
taking values in $\S$ (that is a density $f:\S^{\N}\to\R$), an
infinite-runtime solution of the kind from above is all one can
expect.\par
We denote the set of parameters of HMPs with $d$ hidden states by
$\cH_{d,+}$. By (\ref{eq.dimhdplus}) below, $\cH_{d,+}$ is a
full-dimensional subset of the positive orthant of real affine space
$\R^{d^2+d-1}$. In the form of a theorem, our solution to Problem
\ref{pro.finite} reads as follows.

\begin{Thm}
  \label{t.identification}
  Let $|\S|=2$, $d\le\frac{n+1}{2}$ and $\Prob:\S^n\to[0,1]$ be a
  probability distribution.  There is a an algebraic variety
  $\cN_d\subset\R^{d^2+d-1}$ such that $\dim\cN_d<d^2+d-1$ and an
  algorithmic routine $A$ which, when given $\Prob$ as input, outputs
  \begin{equation*}
    A(\Prob) = 
    \begin{cases}
      \text{'HMP on d hidden states'} & \Prob\in\mathbf{f}_{n,d}((\cH_{d,+}\setminus\cN_d)\\
      \text{'Cannot decide'} & \Prob\in\mathbf{f}_{n,d}(\cN_d)\\
      \text{'No HMP on d hidden states'} & \text{otherwise}
    \end{cases}.
  \end{equation*}
  In the first case, $A$ also outputs the parametrization, which is
  unique up to permutation of hidden states.
\end{Thm}

In the course of proving Theorem~\ref{t.identification}, we provide an
ideal-theoretic characterization of the varieties associated with
finitary processes with arbitrary output alphabets. Based on dimension
arguments, we point out that the varieties of finitary and hidden
Markov processes coincide for binary alphabets. Relationships between
finitary processes and HMPs have been noted already in seminal work on
identification of HMPs
(e.g.~\cite{Blackwell57,Gilbert59,Dharma63a,Dharma63b,Heller65}).
Here we review them from the vantage point of algebraic statistics.
We summarize the corresponding results into the ideal-theoretic
Theorem \ref{t.generators}, which, in turn, is based on the
set-theoretic Lemma \ref{l.generators}. Note that the ideals we
encounter are determinantal in nature; corresponding relationships for
latent variable models have also been noted in
\cite{Bray05,Cueto09,Zwiernik09}.

\paragraph{Short Summary of Contributions.}
\begin{itemize}
\item We provide an algebraic statistical treatment of Problem
  \ref{pro.complete}, based on treating the algebraically more
  convenient Problem \ref{pro.finite}. Thereby, we lay the foundations
  for a treatment that does not require extra assumptions on the
  processes, such as stationarity.
\item We present a generic solution of Problem \ref{pro.finite}, and
  hence of Problem \ref{pro.complete}. The solution is
  \emph{algorithmic in nature} (unlike the only earlier solution from
  \cite{Heller65}).
\item We provide an ideal-theoretic characterization of the varieties
  associated with the probability distributions of finitary processes
  and binary-valued HMPs.
\end{itemize}
All of this is novel, to the best of our knowledge.

\paragraph{Organization of Chapters.}
In section \ref{sec.prelim}, we give the basic definition of an
algebraic statistical model and also the definition of an algebraic
process model, which serves the general purpose to treat stochastic
processes in algebraic statistical settings. In section
\ref{sec.processes}, we give formal definitions of finitary and hidden
Markov processes. In section \ref{sec.models}, we give the definitions
of their algebraic statistical counterparts, the finitary and the
hidden Markov process model. Along with these definitions, we provide
a brief list of fundamental relationships. In section~\ref{sec.dim} we
compute the dimensions of the algebraic varieties associated with
finitary and hidden Markov process models. The crucial observation is
that the varieties of both models coincide for binary-valued output
alphabets.  In section~\ref{sec.hankel} we provide a
Hankel-matrix-based characterization of finitary models hence also of
binary-valued hidden Markov models. The ideal-theoretic formulation of
this is Theorem \ref{t.generators}.  In section~\ref{sec.algorithm} we
present the algorithm on which Theorem~\ref{t.identification} from
above is based.

\paragraph{Major Notation.}
We denote by $\S^*:=\cup_{t\ge 0}\S^t$ the set of all strings over the
alphabet $\S$ where $\S^0=\{\epsilon\}$ with $\epsilon$ the empty
string. We write $a,b\in\S$ for single letters, $v,w\in\S^*$ for
strings and $vw, va$ etc.~for concatenation of letters and
strings. Throughout this paper, if $v=a_1...a_n\in\S^n$
\begin{equation}
  \label{eq.procstring}
  p_X(v) := \Prob(\{X_1=a_1,...,X_n=a_n\})
\end{equation}
refers to the probability that the stochastic process $(X_t)$
generates the string $v\in\S^n$ (for technical convenience we let
stochastic processes start at $t=1$). We simply write $p=p_X$ if this
cannot lead to confusion. We write $'$ for matrix transposition
throughout. None of our algebraic arguments exceed an elementary
level, see \cite{Cox} for an appropriate textbook.

\section{Algebraic Statistical Models}
\label{sec.prelim}

\begin{Def}
  \label{d.algmodel}
  Following \cite{Pachter}, an \emph{algebraic statistical model} with
  $m$ parameters for strings of length $n$ over an alphabet $\S$ is a
  map
  \begin{equation*}
    \begin{array}{rccc}
      \mathbf{f}: &
      \C^m & \longrightarrow &
      \C^{|\S|^n}\\ &\mathbf{z}=(z_1,...,z_m)
      &\mapsto&\mathbf{f}(\mathbf{z}) = (f_v(z_1,...,z_m))_{v\in\S^n}
    \end{array}
  \end{equation*}
  where $f_v\in\C[Z_1,...,Z_m], v\in\S^n$ are polynomials in the
  indeterminates $Z_1,...,Z_m$ and there is a parameter set
  $\cS\subset\C^m$ (usually $\cS\subset\R^m$) such that for
  $\mathbf{z}\in\cS$
  \begin{equation}
    \begin{array}{rccc}
      p_{\mathbf{z}}: & \S^n & \longrightarrow & [0,1]\\
      & v    & \mapsto         & f_v(\mathbf{z})
    \end{array}
  \end{equation}
  is a probability distribution and such that $\C^m$ is the natural
  extension of the parameter set $\cS$ to a complex affine space. 
\end{Def}

We recall that varieties $V\subset\C^n$ correspond to radical ideals
$I\subset\C[X_1,...,X_n]$ insofar as $V$ is the set of zeros of all
polynomials in $I$ \cite{Cox}.  We also recall that an ideal $I$ is
prime iff $xy\in I$ implies $x\in I$ or $y\in I$ and that, in terms of
the above-mentioned correspondence, prime ideals have irreducible
varieties as counterparts.  $\mathbf{f}(\C^m)$, as the image of a
complex-valued polynomial map is a Boolean combination of varieties
(e.g.~\cite[Th.~3.14]{Pachter}). In particular, its topological
closure $V_{\mathbf{f}}=\overline{\mathbf{f}(\C^m)}$ is an irreducible
algebraic variety in $\C^{|\S|^n}$ which corresponds to the prime
ideal $I_{\mathbf{f}}\subset\C[p_v\mid v\in\S^n]$. We write $p_v$
for indeterminates to stress that they are associated with probability
distributions over strings $v\in\S^n$.  We will write $\Prob$ or
$(p(v))_{v\in\S^n}$ for the points in complex affine space
$\C^{\S^n}$.  Polynomials $g\in I_{\mathbf{f}}$ are referred to as
{\em (model) invariants} and the goal of an algebraic statistical
treatment is usually to characterize or even explicitly list these
invariants. See \cite{Pachter,Drton,Drton07} for related textbooks

\subsection{Algebraic Stochastic Process Models}
\label{ssec.processmodels}

When dealing with stochastic processes $(X_t)$, the auxiliary, helpful
observation is that
\begin{equation}
\label{eq.process1}
p_{X}(a_1...a_m)
= \sum_{b_1...b_{n-m}\in\S^{n-m}}
p_{X}(a_1...a_mb_1...b_{n-m}).
\end{equation}
As a consequence, one can make use of (virtual) indeterminates $p_u$
for strings $u$ of length $m$ shorter than $n$ when dealing
with $\C[p_v,v\in\S^n]$:
\begin{equation}
\label{eq.process}
p_u = \sum_{w\in\S^{n-m}}p_{uw}
\end{equation}
reveals $p_u$ as polynomials in the $p_v,v\in\S^n$. This means in
particular that there is no elimination necessary, which is crucial
for this work.\par We emphasize these facts with a definition.

\begin{Def}[Algebraic Stochastic Process Model]
  \label{d.processmodel}
  A family of algebraic statistical models
  \begin{equation}
    (\mathbf{f}_n: \C^d\longrightarrow\C^{\S^n})_{n\in\N}
  \end{equation}
  is called an {\em algebraic (stochastic) process model} if for all
  $1\le m\le n$ and $u\in\S^m$:
  \begin{equation}
    \label{eq.processmodel}
    \mathbf{f}_m(z)_u = \sum_{w\in\S^{n-m}}\mathbf{f}_n(z)_{uw}.
  \end{equation}
\end{Def}

\subsection{Note on Stationarity}
\label{ssec.stationary}
A process $(X_t)$ which takes
values in $\S$ is stationary iff for all $v\in\S^*$
\begin{equation}
  \label{eq.stationary}
  \sum_{a\in\S}
  p_{X}(va)
  = \sum_{a\in\S}
  p_{X}(av)
\end{equation}
which implies (the more common) $p(v)=\sum_{w\in\S^n}p(wv)$ for all
$n,v$, see \eqref{eq.process1}.  Let $\cX$ be a class of parameterized
processes associated with the process model
$(\mathbf{f}_{\cX,n})_{n\in\N}$. Let
\begin{equation}
  V_{\mathbf{f_{\cX,n}}}=V(I_{\mathbf{f}_{\cX,n}})
\end{equation}
be the variety associated with the string length $n$ probability
distributions.  Let $\langle f_j,j\in J\rangle$ be, as usual, the
ideal generated by polynomials $f_j,j\in J$ and $+$ denote addition of
ideals. The stationary distributions in $V_{\mathbf{f_{\cX,n}}}$ then
give rise to the subvariety
\begin{equation}
  \label{eq.statnull}
  V(I_{\mathbf{f}_{\cX,n}}+\langle \sum_{a\in\S}p_{va}-\sum_{a\in\S}p_{av}, v\in\S^{n-1}\rangle).
\end{equation}
This, unless the processes $\cX$ are stationary by definition
establishes that stationary processes form a lower-dimensional
subvariety in $V_{\mathbf{f}_{\cX,n}}$.\\

The extent to which earlier work depends on stationarity often remains
unclear: for \cite{Vidyasagar09}, for example, this is difficult to
determine, whereas \cite{Finesso10} base their approach on
Kullback-Leibler divergence computations, which is definitely only
possible in case of stationarity. As above-mentioned, \cite{Heller65}
is the only contribution that clearly does not depend on
stationarity.\par
Stationarity has geometric implications: by \eqref{eq.statnull},
stationary HMPs only form a null set among all HMPs. Stationarity also
has technical advantages. For example, it introduces certain
symmetries among row and column conditions in the Hankel matrices,
which is discussed in section~\ref{sec.hankel}, see
Remark~\ref{rem.stationarity}.\par In practical applications, it is
very often essential to assume that processes are not stationary.
This becomes evident in particular in application domains where HMPs
or their close derivatives have established ``gold standards'', for
example speech recognition~\cite{Rabiner89}, protein classification
(through profile HMMs) \cite{Durbin}, gene finding \cite{Burge87} and
gene expression time-course analysis
\cite{Hafemeister11}. Therefore a general treatment of HMP
identification is certainly desirable. \par

\section{Processes}
\label{sec.processes}

\subsection{Finitary Processes}
\label{ssec.mmm}

Finitary processes emerged in the above-mentioned early work on HMP
identification
\cite{Blackwell57,Gilbert59,Dharma63a,Dharma63b,Dharma65,Heller65} and
have remained a core concept also in recent work on
identifiability~\cite{Vidyasagar09,Finesso10}.  Finitary processes
were later also referred to as {\em linearly dependent} \cite{Ito92},
{\em observable operator models} \cite{Jaeger00} or as {\em
  finite-dimensional} \cite{Faigle07,Schoenhuth07}. In their possibly
most prevalent application they served to determine equivalence of
hidden Markov processes (HMPs) in 1992 \cite{Ito92} whose exponential
runtime algorithm was later improved to polynomial
runtime~\cite{Schoenhuth11}.

\begin{Def}[Finitary Process]
\label{d.finitary}
A stochastic process $(X_t)$ is said to be \emph{finitary} iff there
are matrices $T_a\in\R^{d\times d}$ for all $a\in\S$ with
$(\sum_{a\in\S}T_a)\mathbf{1}=\mathbf{1}$ (\/that is
$(\sum_{a\in\S}T_a)$ has unit row sums) and a vector $\pi\in\R^d$ whose
entries sum up to one ($\pi'\mathbf{1}=1$) such that
\begin{equation}
\label{eq.finitary}
\Prob (\{X_1=a_1,...,X_n=a_n\}) = \pi'T_{a_1}\cdot\ldots\cdot T_{a_n}\mathbf{1}
\end{equation}
where $\mathbf{1}=(1,...,1)'\in\R^d$ is the vector of all ones. The
parametrization $((T_a)_{a\in\S},x)$ is referred to as
\emph{$d$-dimensional} in case of $\pi\in\R^d$ and $T_a\in\R^{d\times
  d}$ for all $a\in\S$.
\end{Def}

It is an immediate observation that a finitary process which admits a
$d$-dimensional parametrization also admits a parametrization of
dimension $d+1$. 

\begin{Def}[Rank of a Finitary Process]
\label{d.finitaryrank}
The \emph{rank} of a finitary process $(X_t)$ is the minimal dimension
of a parametrization that it admits.
\end{Def}

We conclude by providing a condition that is necessary for rank $d$
finitary processes. For further reference, we use the notation
\begin{equation}
  \label{eq.tva1an}
  T_v := T_{a_1}T_{a_2}\ldots T_{a_{n-1}}T_{a_n}\in\R^{d\times d} 
\end{equation}
for any $v=a_1\ldots a_n\in \Sigma^n$.  

\begin{Prop}
\label{p.finitarynecessary}
Let $(X_t)$ be a finitary process of rank $d$ and let $n\in\N$ be
an arbitrary integer. Then it holds that
\begin{equation}
\rk [p_X(v_iw_j)]_{1\le i,j\le n} \le d
\end{equation}
for all choices of strings $v_1,...,v_n,w_1,...,w_n\in\S^*$.
\end{Prop}

\begin{proof}
Let $((T_a)_{a\in\S},\pi)$ be a $d$-dimensional parametrization of
$(X_t)$. We observe that
\begin{equation}
p_X(v_iw_j) = \langle\pi'T_{v_i},T_{w_j}\mathbf{1}\rangle.
\end{equation}
Since $\pi'T_{v_i}\in\R^{1\times d},T_{w_j}\mathbf{1}\in\R^{d\times 1}$
the claim becomes obvious.\qed
\end{proof}

\subsection{Hidden Markov Processes}
\label{ssec.hmm}

\begin{Def}[Hidden Markov process]
\label{d.hmp}
A \emph{hidden Markov process (HMP)} $(X_t)$ on $d$ hidden states
which takes values in $\Sigma$ is parametrized by a tuple $\Theta=(M,
E, \pi)$ where
\begin{enumerate}
\item $M=[m_{s\bar{s}}]\in\R^{d\times d}$ is a non-negative \emph{transition
  probability matrix} with unit row sums $\sum_{\bar{s}=1}^nm_{s\bar{s}} =1$
  (\/\emph{i.e.} the row vectors of $M$ are probability distributions
  over the hidden states)
\item $E=[e_{sa}]\in\R^{d\times \S}$ is a non-negative {\em emission
  probability matrix} with unit row sums $\sum_{a\in\S}e_{sa}=1$,
  (\emph{i.e.} the row vectors of $E$ are probability distributions over
  $\S$)
\item $\pi$ is an \emph{initial probability distribution} over the
  hidden states
\end{enumerate}
We write 
\begin{equation}
\cH_{d,+}:=\{(M,E,\pi)\mid 
\sum_{\bar{s}}m_{s\bar{s}} = \sum_{a\in\S}e_{sa}= \sum_s\pi_s =1\} 
\subset\R^{d^2+d(|\S|)+d}_+
\end{equation}
for the set of HMP parametrizations. We refer to $\cH_{d,+}$ as the
{\em stochastic} parametrizations.
\end{Def}

\begin{Rem}
If more convenient and not leading to clashes with other indices, we
write $i,j$, instead of $s,\bar{s}$, for hidden states.
\end{Rem}

The naming {\em stochastic} parametrizations is to distinguish them
from more relaxed, complex-valued parameter sets whose definition will
follow. Note that
\begin{equation}
  \label{eq.dimhdplus}
  \dim\cH_{d,+} = d(d-1) + d(|\S|-1) + (d-1) = d^2 + d(|\S|-1) - 1
\end{equation}
which means that $\cH_{d,+}$ can be considered a full-dimensional
subset of $\R^{d^2+d(|\S|-1)-1}$. A HMP $(X_t)$ on $d$ hidden states
as parametrized by $(M,E,\pi)$ proceeds by initially moving to a state
$s\in\{1,...,d\}$ with probability $\pi_s$ and emitting the symbol
$X_1=a$ with probability $e_{sa}$. Then one moves from $s$ to a state
$\bar{s}$ with probability $m_{s\bar{s}}$ and emits the symbol $X_2=b$
with probability $e_{\bar{s}b}$ and so on.\par We further observe that
$M$ decomposes as $M=\sum_{a\in \Sigma} T_a$ where
\begin{equation}
    (T_a)_{s\bar{s}} := e_{sa}\cdot m_{s\bar{s}} 
\end{equation}
which reflect the probabilities to emit symbol $a$ from state $s$ and
subsequently to move on to state $\bar{s}$. In addition, we use the notation
\begin{equation}
O_a:=\diag(e_{1a},...,e_{da})T_a=O_aM.
\end{equation}
which yields $T_a=O_aM$. Correspondingly, we write
$\Theta=(M,(O_a)_{a\in\S},\pi)$ for HMP parametrizations.  In analogy
to finitary process notation, we write
\begin{equation}
  \label{eq.hmptv}
  T_v := T_{a_1}T_{a_2}\ldots T_{a_{n-1}}T_{a_n}=O_{a_1}MO_{a_2}M\ldots O_{a_{n-1}}MO_{a_n}M\in\R^{d\times d} 
\end{equation}
for any $v=a_1\ldots a_n\in \Sigma^n$.  Standard technical
computations then reveal
that, for $v=a_1...a_n\in\S^n$
\begin{equation}
\label{eq.alternativehmm}
\begin{split}
p(v) = \pi'T_{a_1}...T_{a_n}\mathbf{1} = \pi'T_v\mathbf{1},
\end{split}
\end{equation}
where $\mathbf{1}=(1,...,1)'\in\R^d$ is the vector of all ones. 

\begin{Rem}
\label{rem.forwardbackward}
Computation of vectors $\pi'T_v\in\R^{1\times d}$ and
$T_v\mathbf{1}\in\R^{d\times 1}$ reflects the well-known Forward and
Backward algorithms (\emph{e.g.}  \cite{Ephraim02}) for computation of
HMP probabilities. In this respect, entries of these vectors are just
the common Forward and Backward variables. That is
\begin{gather}
(\pi'T_v)'_s = \Pr(S_{n+1} = s\mid X_1=a_1,...,X_n=a_n)\\
(T_v\mathbf{1})_s = \Pr(S_t = s\mid X_{t+1}=a_{1},...,X_{t+n}=a_{n})
\end{gather}
where $(S_t)$ is the (non-observable) Markov process which takes
values in the hidden states $\{1,...,d\}$.
\end{Rem}

(\ref{eq.alternativehmm}) makes it obvious that a HMP on $d$ hidden
states is a finitary processes which admits a $d$-dimensional
parametrization. This allows the following definition.

\begin{Def}[Rank of a Hidden Markov Process]
\label{d.hmmrank}
The \emph{rank} of a hidden Markov process $(X_t)$ is its rank 
as a finitary process.
\end{Def}

The definition gives rise to the following trivial proposition.

\begin{Prop}
\label{p.hmmrank}
A hidden Markov process acting on $d$ hidden states is a finitary 
process of rank at most $d$.
\end{Prop}

\begin{Xmp}[HMPs on $d$ hidden states of rank $d$]
\label{x.fullrankhmm}
Let $\S$ such that $|\S|\ge 2$. Let
$\lambda_1,...,\lambda_d\in(0,1)$ be pairwise different. Consider HMP
parametrizations $(M,E,\pi)$ where
\begin{equation}
M = \Id\in\R^{d\times d},\quad \pi=(\frac1d,...,\frac1d) = \frac1d\mathbf{1}\in\R^d
\end{equation}
where there is $a\in\S$ such that
\begin{equation}
O_a=\diag(\lambda_1,...,\lambda_d).
\end{equation}
The $O_b=\diag(e_{1b},...,e_{db}),b\in\S\setminus\{a\}$ can be chosen
arbitrarily.  Observe that
\begin{equation}
  S(\lambda):=(\mathbf{1}'\Id,...,\mathbf{1}'O_a^{d-1})
  = 
  \begin{pmatrix}
    1 & \cdots & 1\\
    \lambda_1 & \cdots & \lambda_d\\
    \vdots & \ddots & \vdots\\
    \lambda_1^{d-1} & \cdots & \lambda_d^{d-1}
  \end{pmatrix} 
  = [\lambda_j^{i-1}]_{1\le i,j\le d}\in\R^{d\times d}
\end{equation}
forms a Vandermonde matrix and hence is invertible.  Writing
$a^i:=a...a\in\S^i$, it follows that

\begin{equation}
[p(a^{i-1}a^{j-1})]_{1\le i,j\le d} = [\frac1d\cdot\mathbf{1}'O_a^{i-1}O_a^{j-1}\mathbf{1}]_{1\le i,j\le d} 
= \frac1d\cdot S(\lambda)S(\lambda)'\in\R^{d\times d}
\end{equation}
is an invertible matrix. By Proposition~\ref{p.finitarynecessary}, the
hidden Markov process with parametrization $(M,(O_a)_{a\in\S},\pi)$
has rank $d$.
\end{Xmp}

\section{Models}
\label{sec.models}

\subsection{Finitary Models}
\label{ssec.finmodels}

Finitary models $\mathbf{f}_{\cM_d,n}$ are the algebraic statistical
equivalent of {\em finitary processes} that admit $d$-dimensional
parametrizations.

\begin{Def}
\label{d.mmm}
\emph{Finitary models} are 
polynomial maps
\begin{equation}
\label{eq.gnd}
\begin{array}{rccc}
\mathbf{f}_{\cM_d,n}: & \cM_d & \longrightarrow & \C^{\S^n}\\
                 & ((T_a)_{a\in\S}),\pi) & \mapsto & (\pi'T_v\mathbf{1})_{v\in\S^n}.
\end{array}
\end{equation}
where 
\begin{equation}
\label{eq.finitarydef}
\cM_d := \{((T_a)_{a\in\S}),\pi)\in\C^{|\S|d^2+d}\mid 
\sum_{a\in\S}T_a\mathbf{1} = \mathbf{1}\} \cong \C^{|\S|d^2}
\end{equation}
\end{Def}

We write 
\begin{equation*}
  V_{\cM_d,n}:=\overline{\im\mathbf{f}_{\cM_d,n}}
\end{equation*}
for the variety that is associated with $\mathbf{f}_{\cM_d,n}$ and
\begin{equation*}
  I_{\cM_d,n}:=I_{\mathbf{f}_{\cM_d,n}}
\end{equation*}
for the ideal of its invariants. Unlike in the definition of finitary
processes, we do not require that $\pi'\mathbf{1}=1$ which would add
the (technically inconvenient) inhomogeneous invariant $\sum_vp_v = 1$
to $I_{\cM_d,n}$. The relationship
$\sum_aT_a\mathbf{1}=\mathbf{1}$ yields that the family
$(\mathbf{f}_{\cM_d,n})_{n\in\N}$ is an algebraic process model.

\begin{Prop}
\label{p.finitaryprocessmodel}
The family $(\mathbf{f}_{\cM_d,n})_{n\in\N}$ is an algebraic process
model.
\end{Prop}

\begin{proof}
Let $v\in\S^m$. Writing $M:=\sum_aT_a$ we observe
\begin{equation}
\label{eq.finitaryprocess}
\mathbf{f}_{\cM_d,m}(z)_v = \pi'T_v\mathbf{1} 
\stackrel{(\ref{eq.finitarydef})}{=} 
\pi'T_vM^{n-m}\mathbf{1} = \sum_{u\in\S^{n-m}}\mathbf{f}_{\cM_d,n}(z)_{vu}.
\end{equation}
\qed
\end{proof}

By the definition of finitary process models one can further register:

\begin{Prop}
  \label{p.finitaryobservations}
  For all $d,n\in\N$ it holds that $\im\mathbf{f}_{\cM_d,n}\subset\im\mathbf{f}_{\cM_{d+1},n}$.
\end{Prop}

\begin{proof}
   This is because one can extend $d$-dimensional matrices by zero
   entries to obtain a $d+1$-dimensional parametrization and reflects
   that every finitary process with a $d$-dimensional parametrization
   also admits a $d+1$-dimensional parametrization.\qed
\end{proof}

\subsection{Hidden Markov Models}

We obtain an algebraic statistical treatment of HMPs on $d$ hidden
states by allowing that parameters in $M,E$ and $\pi$ are complex. We
write
\begin{equation}
\cH_d:=\{(M,E,\pi)\in\C^{d^2+d|\S|+d} \mid
\sum_{j=1}^nm_{ij}=1,\sum_{a\in\S}e_{ia}=1\}\cong\C^{d^2+d(|\S|-1)}
\end{equation}
for the resulting set of parameters. We still require unit
row sums in both $M$ and $E$, but we do not make any such assumption
for $\pi$. The unit row sum
assumption for $E$ implies that still
\begin{equation}
M = \sum_{a\in\S}T_a\quad\text{where}\quad (T_a)_{s\bar{s}} = e_{sa}m_{s\bar{s}}
\end{equation}
while the unit row sum assumption on $M$ implies that $M\mathbf{1} =
\mathbf{1}$ hence (let $v\in\S^m$ and $m<n$)
\begin{equation}
\label{eq.hmmprocess}
p(v) = \pi'T_v\mathbf{1} = \pi'T_vM^m\mathbf{1} = \sum_{u\in\S^{n-m}}p(vu)
\end{equation}
a relationship which holds for stochastic processes in general. Note 
that
\begin{equation}
  \label{eq.dimhd}
  \dim\cH_d = d^2+d(|\S|-1) \stackrel{(\ref{eq.dimhdplus})}{=} \dim\cH_{d,+} + 1.
\end{equation}
The increase in dimension for $\cH_d$ follows from not requiring that
$\pi$ is a unit vector---in analogy to finitary models we avoid the
non-homogeneous invariant $\sum_vp_v=1$ for technical convenience.\\

\begin{Def}
We recall the notation \eqref{eq.hmptv} and say that 
\label{d.hmm}
\begin{equation}
  \label{eq.algstathmm}
  \begin{array}{rccc}
    \mathbf{f}_{\cH_d,n}: & \cH_d     & \longrightarrow & \C^{|\S|^n}\\
    & (M,(O_a)_{a\in\S},\pi) & \mapsto        & (\pi'T_v\mathbf{1})_{v\in\S^n}
  \end{array}
\end{equation}
is a {\em hidden Markov model} for $d$ hidden states and string length $n$.
\end{Def}

The relationship \eqref{eq.hmmprocess} yields further:

\begin{Prop}
\label{p.hmmprocess}
The family $(\mathbf{f}_{\cH_d,n})_{n\in\N}$ of hidden Markov models
for $d$ hidden states is an algebraic process model.
\end{Prop}

We write 
\begin{equation}
V_{\cH_d,n}:=\overline{\mathbf{f}_{\cH_d,n}(\cH_d)} 
\end{equation}
for the algebraic variety that is associated with
$\mathbf{f}_{\cH_d,n}$ and
\begin{equation*}
I_{\cH_d,n}:=I_{\mathbf{f}_{\cH_d,n}}
\end{equation*}
for the ideal of its invariants. 
\begin{Prop}
\label{p.hmmobservations}
For all $d,n\in\N$:
\begin{itemize}
\item[(a)]
  $\im\mathbf{f}_{\cH_d,n}\;\subset\;\im\mathbf{f}_{\cH_{d+1},n}$
\item[(b)] $\im\mathbf{f}_{\cH_d,n}\;\subset\;\im\mathbf{f}_{\cM_d,n}$.
\item[(c)] $V_{\cH_d,n}\;\subset\; V_{\cM_d,n}$. 
\item[(d)] $I_{\cM_d,n}\;\subset\; I_{\cH_d,n}$.
\end{itemize}
\end{Prop}

While (a) reflects that HMPs on $d+1$ hidden states encompass the HMPs
on $d$ hidden states, (d) translates to the fact that each invariant
of a finitary model applies for the corresponding hidden Markov
model. (d) is a key observation for this work.

\begin{proof} 
  (a) holds because one can extend matrices by zero entries
  thereby obtaining higher-dimensional parametrizations, (b) is
  obvious by the definitions of hidden Markov and finitary process
  models while (c) immediately follows from (b). (c) and (d) finally
  are equivalent, due to elementary algebraic geometric arguments
  \cite{Cox}.\qed
\end{proof}

\section{Dimension}
\label{sec.dim}

\subsection{Finitary Models}
\label{ssec.findim}

In this section we compute the dimension of the variety $V_{\cM_d,n}$
for $n\ge 2d-1$. The key insight to this computation is the following
lemma.

\begin{Lem}
  \label{l.imagedim}
  Let $n\ge 2d-1$ and let
  $\Theta:=((T_a)_{a\in\S},x),\tilde{\Theta}:=((\tilde{T}_a)_{a\in\S},\tilde{x})\in\cM_d$
  be two parameterizations giving rise to finitary processes.
  Consider the following two statements:
  \begin{enumerate}
  \item[(i)]\begin{equation}
      \label{eq.preimage1}
      \mathbf{f}_{\cM_d,n}(\Theta) = \mathbf{f}_{\cM_d,n}(\tilde{\Theta})
    \end{equation}
  \item[(ii)]There exists an invertible linear map $S:\C^d\to\C^d$ such that
    \begin{equation}
      \label{eq.preimage2} 
      S\mathbf{1} = \mathbf{1},\qquad \tilde{x}' = x'S\qquad\text{and}\qquad
      \forall a\in\S:\quad \tilde{T}_{a} = S^{-1}T_aS
    \end{equation}
  \end{enumerate}
  Then $(ii)$ implies $(i)$ and the two statements are equivalent if
  both $\Theta,\tilde{\Theta}$ give rise to 
  processes of rank $d$.
\end{Lem}

\begin{proof}
  While $(ii)\Rightarrow(i)$ is obvious, $(i)\Rightarrow(ii)$ is
  a straightforward generalization of statements
  presented in previous works (e.g.~\cite{Ito92,Jaeger00}) to
  complex-valued parameters $\Theta,\tilde{\Theta}$. \qed
\end{proof}

Lemma \ref{l.imagedim} enables application of a well-known theorem
\cite[Th.~11.12]{Harris} for computing dimensions of varieties.

\begin{Thm}
\label{t.mmmdim}
Let $\mathbf{f}_{\cM_d,n}$ as in Definition \ref{d.mmm} such that $n\ge
2d-1$. Then 
\begin{equation}
\dim V_{\cM_d,n} = \begin{cases} 1 & |\S| =1\\
                                                     (|\S|-1)d^2 + d & |\S|\ge 2
                                         \end{cases}.
\end{equation}
\end{Thm}

\begin{proof}
  The case $|\S|=1$ is trivial: $\im\mathbf{f}_{\cM_n,d}=\C^1$ for all
  $n,d$. For the case $|\S|\ge 2$, we proceed by plugging
  $\cM_d,\mathbf{f}_{\cM_d,n}$ here into $X,\pi$ in
  \cite[Th.~11.12]{Harris}. Therefore we first have to observe that
  $\overline{\mathbf{f}_{\cM_d,n}(\cM_d)}$ is a quasi-projective
  variety, which follows from standard arguments.
  Applying \cite[Th.~11.12]{Harris} then yields
  \begin{equation}
    \label{eq.dimcomp}
    \dim V_{\cM_d,n} =
    \dim\overline{\mathbf{f}_{\cM_d,n}(\cM_d)} = \dim\cM_d - \dim\mathbf{f}_{\cM_d,n}^{-1}(\Theta)
  \end{equation}
  where $\Theta$ is chosen such that
  $\dim\mathbf{f}_{\cM_d,n}^{-1}(\Theta)$ is minimal. Since
  $\dim\cM_d=|\S|d^2$, it remains to show that
  \begin{equation}
    \label{eq.mindim}
    \min_{\Theta\in\cM_d}\dim\mathbf{f}_{\cM_d,n}^{-1}(\Theta)=d(d-1).
  \end{equation}
  Therefore, we first observe that, by Example~\ref{x.fullrankhmm},
  stochastic processes of rank $d$ exist. That is, there is $\Theta$
  such that
  $\mathbf{f}_{\cM_d,n}(\Theta)\not\in\im\mathbf{f}_{\cM_{d-1},n}$. Lemma
  \ref{l.imagedim} then states that
  $\mathbf{f}_{\cM_d,n}(\Theta)=\mathbf{f}_{\cM_d,n}(\bar{\Theta})$ if
  and only if there is an invertible linear map $S\in\C^{d\times d}$
  with $S\mathbf{1}=\mathbf{1}$ by which to transform $\Theta$ into
  $\bar{\Theta}$ as further described in Lemma~\ref{l.imagedim}. This
  yields that the fiber $\mathbf{f}_{\cM_d,n}^{-1}(\Theta)$ has
  dimension equal to that of the space of invertible linear maps $S$
  with $S\mathbf{1}=\mathbf{1}$ which is $d(d-1)$.\par If
  $\mathbf{f}_{\cM_d,n}(\Theta)\in\im\mathbf{f}_{\cM_{d-1},n}(\Theta)$,
  Lemma \ref{l.imagedim} states that the existence of invertible
  linear maps $S$ with $S\mathbf{1}=\mathbf{1}$ that transform
  $\Theta$ into another point
  $\bar{\Theta}\in\mathbf{f}_{\cM_d,n}^{-1}(\Theta)$ is only a
  sufficient condition, which implies
  $\dim\mathbf{f}_{\cM_d,n}^{-1}(\Theta)\ge d(d-1)$. In summary, we obtain
  \eqref{eq.mindim}, which concludes the proof.\qed
\end{proof}

\subsection{Hidden Markov Models}
\label{ssec.hmmdim}

Let $\cH_{d,0}\subset\cH_d$ encompass all parametrizations
$\Theta=(M,(O_a)_{a\in\S},\pi)$ such that
\begin{itemize}
\item $M$ is not invertible, \emph{or}
\item there is no $a\in\S$ such that the eigenvalues of $O_a$ are
  pairwise different.
\end{itemize}
Note that $\mathbf{f}_{\cH_d,n}^{-1}(\cM_{d-1,n})$ are the HMM
parametrizations on $d$ hidden states whose rank is less than $d$. We
set
\begin{equation}
  \label{eq.nullset}
  \cN_d:=\mathbf{f}_{\cH_d,n}^{-1}(\cM_{d-1,n}) \cup \cH_{d,0}. 
\end{equation}

\begin{Lem}
  \label{l.genericfiber}
  $\cN_d$ forms a variety of dimension
  \begin{equation}
    \dim\,\cN_d<\dim\cH_d=d^2+d(|\S|-1)
  \end{equation}
  and for $\Theta=(M,E,\pi)\in \cH_d\setminus\cN_d$ it holds that
  \begin{equation}
    \label{eq.genericfiber}
    \card\mathbf{f}_{\cH_d,n}^{-1}(\mathbf{f}_{\cH_d,n}(\Theta)) = d! < \infty
  \end{equation}
\end{Lem}

The cardinality of the generic fiber in \eqref{eq.genericfiber}
reflects that permutation of the $d$ hidden states yields an
equivalent HMP.

\begin{proof}
  Theorems 3.1 and 3.2 in \cite{Baras92} prove this for stationary
  processes. Our proof consists in observing that the stationarity
  assumption in \cite{Baras92} is not used. Moreover, it is
  straightforward to replace real values by complex values.\qed
\end{proof}

\begin{Rem} 
  \cite{Allman09} provide alternative arguments to prove
  identifiability of stationary HMPs. Although formulated only for
  stationary HMPs in \cite{Allman09}, the arguments can be easily
  extended to non-stationary HMPs and also to complex
  values. \cite{Allman09} particularly focus on generic
  identifiability of HMPs from their distributions over strings of
  length $n<2d-1$ for alphabets $|\S|>2$. As they do not explicitly
  name the generic subsets, application of results from the earlier
  \cite{Baras92} yields a more convenient treatment here.
\end{Rem}

\begin{Cor}
  \label{c.genericfiber}
  As real varieties,
  \begin{equation}
    \dim(\cN_d\cap\cH_{d,+}) < \dim\cH_{d,+} = \dim\cH_{d} - 1
  \end{equation}
  and 
  \begin{equation}
    \card\mathbf{f}_{\cH_d,n}^{-1}(\mathbf{f}_{\cH_d,n}(\Theta)) = d!
  \end{equation}
  for $\Theta\in\cH_{d,+}\setminus\cN_d$.
\end{Cor}

\begin{proof}
  The proof is analogous to that for Lemma~\ref{l.genericfiber}.  The
  reduction in dimension by $1$ for $\cH_{d,+}$ is due to not
  requiring $\sum_{i=1}^d\pi_i=1$ for $\Theta\in\cH_{d}$, see
  (\ref{eq.dimhdplus}).\qed
\end{proof}

\paragraph{Identification Algorithm: Workflow.} 
We pause for a moment and relate the results obtained so far with the
statements of Theorem~\ref{t.identification}. Given a probability
distribution $\Prob:\S^n\to[0,1]$ as input, the algorithm of
Theorem~\ref{t.identification} will proceed in three steps:
\begin{enumerate}
\item Determine whether
  $\Prob\in\im\mathbf{f}_{\cH_d,n}$.
\item If yes, determine $\Theta\in\cH_d$ such that
  $\mathbf{f}_{\cH_d,n}(\Theta)=\Prob$.
\item If $\Theta\in\cH_d\setminus\cN_d$ determine whether $\Theta$
  is real non-negative.
\end{enumerate}
From this outer perspective, Lemma~\ref{l.genericfiber} and
Corollary~\ref{c.genericfiber} are key to performing the third step.
We will provide the ingredients for the first two steps in the
subsequent sections \ref{sec.hankel} and \ref{sec.algorithm}. We
create the necessary link to these sections with the main theorem of
this section.

\begin{Thm}
\label{t.hmmdim}
Let $\mathbf{f}_{\cH_d,n}$ be as in Definition \ref{d.hmm} where $n\ge
2d-1$. Then it holds that
\begin{equation}
\dim V_{\cH_d,n}=\begin{cases} 1 & |\S| =1\\
d^2 + (|\S|-1)d & |\S|\ge 2
                                         \end{cases}.
\end{equation}
\end{Thm}

\medskip

\begin{proof}
  The proof again is an application of \cite[Th.11.12]{Harris}.  Let
  $\Theta=((T_a=O_aM)_{a\in\S},\pi)\in\cH_d\setminus\cN_d$.
  (\ref{eq.genericfiber}) implies that
  \begin{equation}
    \dim\mathbf{f}_{\cH_d,n}^{-1}(\Theta) = 0.
  \end{equation}
  Applying \cite[Th.~11.12]{Harris} in the way of the proof for
  Theorem~\ref{t.mmmdim} yields
  \begin{equation}
    \begin{split}
      \dim V_{\cH_d,n} &=
      \dim\overline{\im\mathbf{f}_{\cH_d,n}}\\ 
      &= \dim\C^{d^2+(|\S|-1)d} - \dim\mathbf{f}_{\cH_d,n}^{-1}(\Theta)\\ 
      &= d^2+(|\S|-1)d - 0.
    \end{split}
  \end{equation}
  \qed
\end{proof}

\paragraph{Binary-Valued HMMs}

In case of a two-letter alphabet $\S$ we find
\begin{equation*}
\dim V_{\cH_d,n} = (|\S|-1)d+d^2 = d+d^2 = d+(|\S|-1)d^2 = \dim V_{\cM_d,n}.
\end{equation*}
Since $V_{\cH_d,n}\subset V_{\cM_d,n}$ and both varieties are
irreducible, $V_{\cH_d,n}$ and $V_{\cM_d,n}$ coincide, which is a
standard conclusion from algebraic geometry~\cite[Prop.~10,
p.~463]{Cox}. Therefore, we obtain the following key insight.

\begin{Cor}
  \label{c.binaryhmm}
  If $|\S|=2$
  \begin{equation}
    V_{\cH_d,n}=V_{\cM_d,n}.
  \end{equation}
  \qed
\end{Cor}

\section{Invariants}
\label{sec.hankel}

Computation of invariants for finitary models is made possible by a
{\em Hankel matrix} based characterization of finitary processes,
corollaries of which will also shed light on the relationship $n\ge
2d-1$ in the formulation of Problem~\ref{pro.finite}.

\subsection{The Hankel Matrix}
\label{ssec.hankel}

\begin{Def}
  \label{d.probfunction}
  A string function $p:\S^*\to\C$ such that 
  \begin{equation}
    \label{eq.probfunction}
    \forall v\in\S^*:\;\sum_{a\in\S}p(va)=p(v)
  \end{equation}
  is called a {\em process function}.
\end{Def}

$\sum_ap(va)=p(v)$ implies $\sum_{u\in\S^m}p(vu) = p(v)$ for all
$m\in\N$ which parallels the definition of a process model. By
standard arguments, string functions $p:\S^*\to\C$ are associated with
stochastic processes if and only if
\begin{equation}
  \forall v\in\S^*:\;\sum_{a\in\S}p(va)=p(v), \quad 
  \sum_{a\in\S}p(a)=1\quad\text{and}\quad p(\S^*)\subset[0,1].
\end{equation}
Omitting $\sum_ap(a)=1$, $p(\S^*)\subset[0,1]$ in the definition of
process function is for compatibility with algebraic process models,
see Def.~\ref{d.processmodel}.

\begin{Def}
  \label{d.hankel}
  Let $p:\S^*\to\C$ be a string function. 
  \begin{itemize}
  \item
    \begin{equation}
      \label{eq.hankel}
      \cP_p := [p(vw)_{v,w\in \Sigma^*}]\in\C^{\Sigma^*\times\Sigma^*}
    \end{equation}
    is called the {\em Hankel matrix} of $p$ (also called {\em
      prediction matrix} in case of a process function $p$, see
    e.g.~\cite{Schoenhuth07}). 
  \item We define 
    \begin{equation}
      \rk p := \rk\cP_p
    \end{equation}
    to be the {\em rank} of the string function $p$. 
  \item In case of $\rk p<\infty$ the string function $p$ is said
    to be {\em finitary}.
  \end{itemize}
\end{Def}

\begin{Xmp} 
  \label{x.hankel}
  Let $p:\S^*\to\C$ be a string function over the binary alphabet
  $\S=\{0,1\}$. Using lexicographical order on finite strings, the Hankel
  matrix is
  \begin{equation*}
    \cP_p = 
    \begin{pmatrix} 
      p(\epsilon) & p(0) & p(1)   & p(00)   & p(01)   & p(10)   & p(11)   & \hdots \\
      p(0)  & p(00)  & p(01)  & p(000)  & p(001)  & p(010)  & p(011)  & \hdots \\
      p(1)  & p(10)  & p(11)  & p(100)  & p(101)  & p(110)  & p(111)  & \hdots \\
      p(00) & p(000) & p(001) & p(0000) & p(0001) & p(0010) & p(0011) & \hdots \\
      p(01) & p(010) & p(011) & p(0100) & p(0101) & p(0110) & p(0111) & \hdots \\
      p(10) & p(100) & p(101) & p(1000) & p(1001) & p(1010) & p(1011) & \hdots \\
      p(11) & p(110) & p(111) & p(1100) & p(1101) & p(1110) & p(1111) & \hdots \\
      \vdots& \vdots & \vdots & \vdots  & \vdots  & \vdots  & \vdots  & \ddots
    \end{pmatrix}
  \end{equation*}
  See also \cite{Finesso10} for examples.
\end{Xmp} 

\begin{Xmp}[Rank 1 Hankel matrices: i.i.d.~processes]
Let $(X_t)$ be an i.i.d.~stochastic process taking values in
$\S$. That is, there are $\rho_a\in[0,1],a\in\S$ with $\sum_a\rho_a=1$
such that
\begin{equation}
p_X(a_1...a_n)=\rho_{a_1}\cdot...\cdot\rho_{a_n}
\end{equation}
for all $a_1...a_n\in\S^*$. We observe that $\rk\cP_{p_X} = 1$ in that case.
In fact, $\rk\cP_{p_X} = 1$ is a characterization of i.i.d.~processes.
\end{Xmp}

\begin{Xmp}
Revisiting Example~\ref{x.fullrankhmm} (\/there $\Sigma$ was
$\{a,b\}$) yields that the finite submatrix 
\begin{equation}
  \begin{pmatrix}
    p(\epsilon) & p(a) & \cdots & p(a^{d-1})\\
    p(a)        & p(aa) & \cdots & p(aa^{d-1}=a^d)\\
    \vdots      & \vdots & \ddots & \vdots\\
    p(a^{d-1})   & p(a^{d-1}a=a^d) & \cdots & p(a^{d-1}a^{d-1}=a^{2d-2})
  \end{pmatrix}\in[0,1]^{d\times d}
\end{equation}
of $\cP_{p_X}$, as an invertible matrix, has rank $d$.
\end{Xmp}

For finitary processes one may ask if their rank as process (see
Definition~\ref{d.finitaryrank}) agrees with their rank as string
function.

Generalizing \cite{Faigle07} one can show that this is the case and
therefore $\cM_d$ consists precisely of parameterizations with process
functions of rank $\le d$.

\begin{Thm}
\label{t.oom}
Let $p:\S^*\to\C$ be a process function.
Then the following conditions are equivalent.
\begin{enumerate}
\item[(i)] $p$ is finitary of rank at most $d$.
\item[(ii)] There exist vectors $x,y\in\C^d$ as well as matrices
            $T_a\in\C^{d\times d}$ for all $a\in\S$ such that
\begin{equation}
\label{eq.oom}
\forall a_1...a_n\in\S^*:\quad p(a_1...a_n) = x'T_{a_1}...T_{a_n}y
\quad\text{and}\quad (\sum_{a\in\S}T_a)y=y.
\end{equation}
\item[(iii)] There exists a vector $x\in\C^d$ as well as matrices
            $T_a\in\C^{d\times d}$ for all $a\in\S$ such that
\begin{equation}
\label{eq.oom1}
\forall a_1...a_n\in\S^*:\quad p(a_1...a_n) = x'T_{a_1}...T_{a_n}\mathbf{1}
\quad\text{and}\quad (\sum_{a\in\S}T_a)\mathbf{1}=\mathbf{1}.
\end{equation}
where $\mathbf{1}=(1,...,1)'\in\C^d$ is the vector of all ones.
\end{enumerate}
\end{Thm}

\begin{proof}
  $(ii)\Leftrightarrow (iii)$ (where $(iii)$ trivially implies $(ii)$)
  follows from the observation that, given an invertible linear map
  $S:\C^d\to\C^d$ such that $S\mathbf{1}=y$ yields
  \begin{equation}
    \label{eq.transform}
    x'T_{a_1}...T_{a_n}y = x'SS^{-1}T_{a_n}SS^{-1}...SS^{-1}T_{a_1}SS^{-1}y 
    = \tilde{x}'\tilde{T}_{a_n}...\tilde{T}_{a_1}\mathbf{1}
  \end{equation}
  where $\tilde{T}_{a_i}=S^{-1}T_{a_i}S,\tilde{x}=S'x$.
  $(iii)\Rightarrow(i)$ is because the arguments from
  \cite{Jaeger00,Schoenhuth07} work for arbitrary fields.
  $(i)\Rightarrow(ii)$ follows because the arguments from
  \cite{Jaeger00,Schoenhuth07} do not require $\sum_{v\in\S^n}p(v)=1$,
  which is missing here, for a proof.\qed
\end{proof}

\paragraph{Finite Algebraic Relationships}
\label{ssec.finiteness}

In the following, we write
\begin{equation}
\cP_{p,m,n} := [p(vw)]_{|v|\le m,|w|\le n}\in\C^{\frac{(|\S|^{m+1}-1)}{|\S|-1}\times\frac{(|\S|^{n+1}-1)}{|\S|-1}}.
\end{equation}
for the upper left submatrices of $\cP$ which refer to prefixes and
suffixes of length at most $m$ and $n$. Well-known arguments
(e.g.~\cite[Lemma~2.4]{Schoenhuth07}) show that
\begin{equation}
  \label{eq.dimcheck}
  \rk\cP_p = \rk\cP_{p,d-1,d-1}
\end{equation}
for a process function $p$ of rank $\le d$.  It follows that a process
function of rank $\le d$ is uniquely determined by the values
\begin{equation}
  \label{eq.unique}
  p(v), \quad |v| = 2d-1
\end{equation}
which applies for $d$-state HMPs. Combining this with Lemma
\ref{l.genericfiber} yields that HMPs are generically identifiable
from their string-length $2d-1$ probabilities. 

\begin{Rem}
  \cite{Allman09} demonstrate that for $|\S|>2$, HMPs are generically
  identifiable already from distributions over strings of length
  smaller than $2d-1$.  However, hidden Markov processes are no longer
  uniquely determined by their distributions on strings of length
  smaller than $2d-1$.  We conjecture that we could, with some extra
  work, also employ \cite{Allman09} for our arguments. However, the
  bounds provided by \cite{Allman09} on string length agree with the
  ones in use here for $|\S|=2$ in any case. To date, for binary-valued
  alphabets, $2d-1$ is the lowest bound presented in the literature.
\end{Rem}

\begin{Rem}[Stationarity]
  \label{rem.stationarity}
  Let $p$ represent a stochastic process and let $(\cP_p)_v:\S^*\to\C$ be
  the $v$-row in $\cP_p$ (that is $(\cP_p)_v(w)=p(vw)$)
  resp.~$(\cP_p)^w:\S^*\to\C$ be the $w$-column of $\cP_p$ (that is
  $(\cP_p)^w(v)=p(vw)$). Because $p$ is a process, we have
  \begin{equation}
    \label{eq.columncond}
    \sum_{a\in\S}(\cP_p)^{wa}=(\cP_p)^w
  \end{equation}
  which is a reformulation of the recurring theme \eqref{eq.process}.
  In case that $p$ is a stationary process, \eqref{eq.stationary} translates to
  \begin{equation}
    \label{eq.stathankel}
    \sum_{a\in\S}(\cP_p)_{av}=(\cP_p)_v.
  \end{equation}
  This introduces a ``symmetry'' in $\cP_p$ insofar as
  \eqref{eq.stathankel} is the condition for the rows that is
  analogous to the column condition \eqref{eq.columncond}.
\end{Rem}

To summarize, Theorem~\ref{t.oom} states that the finitary processes
of rank $\le d$ are precisely the ones whose process functions give
rise to Hankel matrices of rank at most $d$.  This translates to the
fact that $\rk p\le d$ if and only all $(d+1)\times(d+1)$-minors of
$\cP_p$ are zero, which yields a polynomial characterization of rank
$\le d$ processes.\par This characterization, however, requires usage
of probabilities for strings of arbitrary length. Since we aim at
obtaining polynomial equations for probabilities of strings of a fixed
length $n$ alone (where $n\ge 2d-1$), we need to collect further
insights. Therefore, immediately note that \eqref{eq.process} reveals
probabilities of strings of length $m<n$ as sums of probabilities of
length-$n$ strings. We will demonstrate in the next section how to
avoid probabilities for strings of length $m>n$.

\subsection{Ideals and Varieties}
\label{sec.maintheorem}

Let
\begin{equation}
  R := \C[p_v\mid v\in\S^*]
\end{equation}
be the polynomial ring with (infinitely many) indeterminates $p_v$. 
Let further
\begin{equation}
  \label{eq.subring}
  R_n := \C[p_v\mid v\in\S^n]
\end{equation}
\eqref{eq.process} reveals that $R_m$ can be regarded as a subring of
$R_n$ for $m<n$, which is crucial in the following.\par We define
\begin{equation}
  \cP_R := [p_{vw}]_{v,w\in \Sigma^*}\in R^{\Sigma^*\times\Sigma^*}
\end{equation}
as a matrix of Hankel type whose entries are indeterminates
$p_{vw}$, which is analogous to $\cP_p$. As for $\cP_p$, we also write
\begin{equation}
  \cP_{R,m,n} := [p_{vw}]_{|v|\le m,|w|\le n}\in R^{\frac{(|\S|^{m+1}-1)}{|\S|-1}\times\frac{(|\S|^{n+1}-1)}{|\S|-1}}.
\end{equation}

Let
\begin{multline}
  \label{eq.idealdef1}
    I_{d+1,n} :=
    \langle\;\;f \mid f \text{ (d+1)-minor of }\cP_{R,\lfloor\frac{n}{2}\rfloor,\lceil\frac{n}{2}\rceil}\;\;\rangle
    + \langle\;\;f \mid f \text{ (d+1)-minor of }\cP_{R,\lceil\frac{n}{2}\rceil,\lfloor\frac{n}{2}\rfloor}\;\;\rangle 
\end{multline}
$I_{d+1,n}$ is the ideal of all $(d+1)$-minors in either
$\cP_{R,\lfloor\frac{n}{2}\rfloor,\lceil\frac{n}{2}\rceil}$ or
$\cP_{R,\lceil\frac{n}{2}\rceil,\lfloor\frac{n}{2}\rfloor}$. Let
\begin{equation}
  \label{eq.idealdef2}
    J_{d,n} := \langle\;\;g\mid\; g\text{ d-minor of }\cP_{R,d-1,d-1}\;\;\rangle
\end{equation}
$J_{d,n}$ is the ideal of all $d$-minors in $\cP_{R,d-1,d-1}$. Due to the
comment following \eqref{eq.subring}, one can view both $I_{d+1,n}$ and $J_{d,n}$
as ideals of $R_n$. 

\begin{Rem}
  \label{rem.prime}
  Elementary insights \cite{Bruns} point out that determinantal ideals
  are prime if matrix entries represent independent
  indeterminates. This is not the case here---as just outlined,
  $p_u=\sum_{w\in\S^{n-m}}p_{uw}$ for $|u|=m<n=|uw|$ reveals $p_u$,
  for $|u|=m<n$, as a sum of indeterminates referring to strings of
  length $n$. Indeed, computations with Bertini \cite{Bertini} confirm
  that $I_{3,4}$, for example, is not prime.
\end{Rem}

Let $\rad I$ be the radical of $I$ and $I:J$ the quotient ideal of $I$
with respect to $J$.

\begin{Thm}
  \label{t.generators}
  Let $n\ge 2d-1$. Then
  \begin{equation*}
    I_{\cM_d,n} = \rad I_{d+1,n} : J_{d,n}.
  \end{equation*}
  For $|\S|=2$, also
  \begin{equation}
    I_{\cH_d,n}=\rad I_{d+1,n} : J_{d,n}.
  \end{equation}
\end{Thm}
\medskip

Theorem~\ref{t.generators} provides an ideal-theoretic
characterization of the variety associated with the finitary model
$\mathbf{f}_{\cM_d,n}$. If $|\S|=2$, this also applies for the hidden
Markov model $\mathbf{f}_{\cH_d,n}$, due to
Corollary~\ref{c.binaryhmm}.

\begin{Rem}
  As pointed out in Remark~\ref{rem.prime}, the quotient operation is
  necessary.  However, it remains an open problem whether the radical
  operation is.  Macaulay \cite{Macaulay} computations reveal that it
  is not for $d=2,n=3$. Macaulay and Bertini computations also confirm
  Theorem~\ref{t.generators} in terms of dimension computations.
\end{Rem}

The proof of Theorem~\ref{t.generators} is based on a set-theoretic
lemma which makes use of the insights assembled in the earlier
chapters. For the following, we recall the notation 
\begin{equation}
  T_v = T_{a_1}...T_{a_n}\quad\text{for }v=a_1...a_n
\end{equation}
see \eqref{eq.tva1an}.

\begin{Lem}
\label{l.generators}
Let $n\ge 2d-1$ and $(p(v))_{v\in\S^n}\in\C^{\S^n}$.  The
following statements are equivalent:
\begin{enumerate}
\item[(i)]
\begin{equation}
(p(v))_{v\in\S^n}\in\im\mathbf{f}_{\cM_d,n}\setminus\im\mathbf{f}_{\cM_{d-1},n}
\end{equation}
\item[(ii)]
\begin{equation}
\label{eq.generators}
\rk\cP_{p,d-1,d-1} = \rk \cP_{p,\lfloor\frac{n}{2}\rfloor, \lceil\frac{n}{2}\rceil } = 
\rk\cP_{p,\lceil\frac{n}{2}\rceil, \lfloor\frac{n}{2}\rfloor } = d
\end{equation}
\end{enumerate}
In case of \eqref{eq.generators}, one can
choose parameters for $(p(v))_{v\in\S^n}$ by determining an invertible
submatrix 
\begin{equation}
  \label{eq.V}
  V=[p(v_iw_j)]_{1\le i,j\le d}\in\C^{d\times d} 
\end{equation}
from
$\cP_{p,d-1,d-1}$ and setting
\begin{align}
  \label{eq.x}
  x' & := (p(w_1),...,p(w_{d}))\\
  \label{eq.y}
  y & := V^{-1}\begin{pmatrix}p(v_1)\\ \vdots\\ p(v_{d})\end{pmatrix}\\
  \label{eq.Ta}
  T_a & := V^{-1}W_a := V^{-1}[p(v_iaw_j)]_{1\le i,j\le d}
\end{align}
which yields that $p(v) = x'T_vy$ so that we obtain
\begin{equation}
  p(v) = \pi\tilde{T_v}\mathbf{1}
\end{equation}
by further application of Theorem \ref{t.oom}.  Note that
probabilities in $W_a$ may refer to strings $v_iaw_j$ of length up to
$2d-1$. This explains the necessity of the assumption $n\ge 2d-1$.  
\end{Lem}
\bigskip

$n\ge 2d-1$ implies that $d-1<\lceil\frac{n}{2}\rceil$. Writing
$A\subsetneq B$ for a submatrix $A$ of a matrix $B$ which is strictly
smaller than $B$ shows that
\begin{equation}
\label{eq.pd1d1}
\cP_{p,d-1,d-1}\;\subsetneq\;\cP_{p,\lfloor\frac{n}{2}\rfloor, \lceil\frac{n}{2}\rceil }
\quad\text{and}\quad
\cP_{p,d-1,d-1}\;\subsetneq\;\cP_{p,\lceil\frac{n}{2}\rceil, \lfloor\frac{n}{2}\rfloor }
\end{equation}

\begin{Xmp}
\label{x.hankel42}
Let $n=3,d=2$ and $\S=\{0,1\}$. Hence $\lceil\frac{n}{2}\rceil = 2$
and $\lfloor\frac{n}{2}\rfloor = 1$ such that we have
\begin{equation*}
\cP_{p,\lceil\frac{n}{2}\rceil,\lfloor\frac{n}{2}\rfloor} = \cP_{p,2,1} =
  \begin{pmatrix} 
      p(\epsilon) & p(0) & p(1)   \\
       p(0)  & p(00)  & p(10)  \\
       p(1)  & p(01)  & p(11)  \\
       p(00) & p(000) & p(100) \\
       p(01) & p(001) & p(101) \\
       p(10) & p(010) & p(110) \\
       p(11) & p(011) & p(111) \\
  \end{pmatrix}\in\C^{7\times 3}
\end{equation*}
\begin{multline*}
\cP_{p,\lfloor\frac{n}{2}\rfloor,\lceil\frac{n}{2}\rceil} = \cP_{p,1,2} =\\
  \begin{pmatrix} 
       p(\epsilon) & p(0) & p(1)   & p(00)   & p(01)   & p(10)   & p(11)   \\
       p(0)  & p(00)  & p(10)  & p(000)  & p(010)  & p(100)  & p(110)  \\
       p(1)  & p(01)  & p(11)  & p(001)  & p(011)  & p(101)  & p(111)  \\
  \end{pmatrix}\in\C^{3\times 7}
\end{multline*}
and
\begin{equation*}
\cP_{p,d-1,d-1} = \cP_{p,1,1} =
  \begin{pmatrix} 
      p(\epsilon) & p(0) & p(1)   \\
       p(0)  & p(00)  & p(10)  \\
       p(1)  & p(01)  & p(11)  \\
  \end{pmatrix}\in\C^{3\times 3}.
\end{equation*}
We recall the relationship $p(v)=\sum_{w\in\S^{3-|v|}}p(vw)$
(\ref{eq.probfunction}). For example,
\begin{eqnarray*}
p(00) & = & p(000) + p(001)\\
p(1) & = & p(100) + p(101) + p(110) + p(111)\\
p(\epsilon) & = & p(000) + p(001) + ... + p(110) + p(111)
\end{eqnarray*}
which yields expressions in strings of length $n=3$ only. 
As $\cP_{p,d-1,d-1}$ is a submatrix of both
$\cP_{p,\lfloor\frac{n}{2}\rfloor,\lceil\frac{n}{2}\rceil }$ and 
$\cP_{p,\lceil\frac{n}{2}\rceil,\lfloor\frac{n}{2}\rfloor}$
we can decompose (\ref{eq.generators}) into
\begin{eqnarray}
\rk\cP_{p,1,1} & \ge & 2\label{eq.conda}\\
\rk \cP_{p,1,2} & \le & 2\label{eq.condb}\\
\rk\cP_{p,2,1} & \le & 2\label{eq.condc}.
\end{eqnarray}
\end{Xmp}
\medskip

{\em Proof of Lemma~\ref{l.generators}.}
(i) $\Rightarrow$ (ii): Let $(p(v))_{v\in\S^n}$ be in the image of
$\mathbf{f}_{\cM_d,n}$, but not in the image of $\mathbf{f}_{\cM_{d-1},n}$.
Combining Theorem~\ref{t.oom} with \eqref{eq.dimcheck} 
reveals that
\begin{equation}
  d = \rk p = \stackrel{{\rm Th.\ref{t.oom}}}{=} \rk \cP_p
  \stackrel{\eqref{eq.dimcheck}}{=} \rk\cP_{p,d-1,d-1} 
\end{equation}
where the second equation is just the definition of the rank of a
string function. Since
$\rk\cP_{p,d-1,d-1}\le\rk\cP_{p,\lfloor\frac{n}{2}\rfloor,\lceil\frac{n}{2}\rceil},
\rk\cP_{p,\lceil\frac{n}{2}\rceil,\lfloor\frac{n}{2}\rfloor}\le\rk\cP_p$
(see (\ref{eq.pd1d1})), we obtain the claim.\\

(ii) $\Rightarrow$ (i): Let $P:=(p(u))_{u\in\S^n}\in\C^{\S^n}$ such
that (\ref{eq.generators}) applies. By Theorem~\ref{t.oom},
$P\in\mathbf{f}_{\cM_{d-1},n}$ would imply $\rk\cP_{p,d-1,d-1}\le
d-1$, a contradiction! In order to show that
$P\in\im\mathbf{f}_{\cM_d,n}$ we will demonstrate that determining
$V,x,y,(T_a)_{a\in\S}$ according to
\eqref{eq.V},\eqref{eq.x},\eqref{eq.y},\eqref{eq.Ta} yields
\begin{eqnarray}
  p(u) & = & x'T_{u}y\quad\text{for all }u\in\S^*\label{eq.oomrepr1}\\
  (\sum_{a\in\S}T_a)y & = & y\label{eq.oomrepr2}.
\end{eqnarray}
Applying $(ii)\Rightarrow(iii)$ from Theorem~\ref{t.oom} to $x,y$
and $T_a,a\in\S$ then proves the claim.\\

The proof concludes by means of the following two elementary
sublemmata \ref{sl.generators.1}, \ref{sl.generators.2}.

\begin{Lem} 
  \label{sl.generators.1}
  Let $v=a_1...a_m\in\S^*$ such that
  $|v|=m\le\lceil\frac{n}{2}\rceil$. Then
  \begin{equation}
    \label{eq.generators.1}
    x'T_v = (p(vw_1),...,p(vw_d)).
  \end{equation}
\end{Lem}

{\em Proof of Lemma~\ref{sl.generators.1}.}
By induction on $|v|$, we obtain a proof by showing
\begin{equation}
  (p(vw_1),...,p(vw_d)) T_a = (p(vaw_1),...,p(vaw_d))
\end{equation}
for all $v\in\S^*,a\in\S$ with $|v|<\frac{n}{2}$. Therefore, $|w_j|\le
d-1<\frac{n}{2}$ implies $|aw_j|\le \lceil\frac{n}{2}\rceil$. Hence
both $|va|,|aw_j|\le\lceil\frac{n}{2}\rceil$. Furthermore, $|v|<n/2$
implies $|v|\le\lfloor\frac{n}{2}\rfloor$ and $\rk\cP_{p,d-1,d-1} =
\cP_{p,\lfloor\frac{n}{2}\rfloor, \lceil\frac{n}{2}\rceil }$ from
(\ref{eq.generators}) implies that the $v$-row $(\cP_p)_v$ in
$\cP_{p,\lfloor\frac{n}{2}\rfloor,\lceil\frac{n}{2}\rceil }$ is
contained in the span of the rows $(\cP_p)_{v_i}$, by choice of the
$v_i$ (\ref{eq.V}). Accordingly, we determine $\alpha_i,i=1,...,d$
such that $(\cP_p)_v = \sum_{i=1}^{d}\alpha_i(\cP_p)_{v_i}$ which, by
definition of
$\cP_{p,\lfloor\frac{n}{2}\rfloor,\lceil\frac{n}{2}\rceil}$, yields
$p(vw) = \sum_{i=1}^d\alpha_ip(v_iw)$ for all
$w,|w|\le\lceil\frac{n}{2}\rceil$.  As $|aw_j|\le\frac{n}{2}$ for all
$j=1,...,d$, we obtain in particular
\begin{equation}
  \label{eq.insight}
  (p(vaw_1),...,p(vaw_d)) = \sum_{i=1}^d\alpha_i(p(v_iaw_1),...,p(v_iaw_d)).
\end{equation}
This is the key insight. We finally compute
\begin{equation}
  \begin{split}
    (&p(vw_1),...,p(vw_d))T_a = \sum_{i=1}^d\alpha_i(p(v_iw_1),...,p(v_iw_d))T_a\\
    &=\sum_{i=1}^d\alpha_i(p(v_iw_1),...,p(v_iw_d))V^{-1}W_a = \sum_{i=1}^d\alpha_ie_i'W_a\\
    &=\sum_{i=1}^d\alpha_i(p(v_iaw_1),...,p(v_iaw_d)) \stackrel{(\ref{eq.insight})}{=}
    (p(vaw_1),...,p(vaw_d)).
  \end{split}
\end{equation}
\qed\\
  
\begin{Lem}
  \label{sl.generators.2}
  For all $v,w\in\S^*$ such that
  $|v|\le\lceil\frac{n}{2}\rceil,|w|\le\lfloor\frac{n}{2}\rfloor$ 
  (which implies $|vw|\le n$):
  \begin{equation}
    (p(vw_1),...,p(vw_d))T_wy = p(vw).
  \end{equation}
\end{Lem}

{\em Proof of Lemma~\ref{sl.generators.2}.}  We do this by induction
on $|w|$, starting with $|w|=0$, that is $w=\epsilon$ and
$T_{\epsilon}=V^{-1}W_{\epsilon}=Id$.  Due to $\rk\cP_{p,d-1,d-1} =
\rk \cP_{p,\lceil\frac{n}{2}\rceil,\lfloor\frac{n}{2}\rfloor }$, by
(\ref{eq.generators}), the row $(\cP_p)_v$ in
$\cP_{p,\lceil\frac{n}{2}\rceil,\lfloor\frac{n}{2}\rfloor}$ is
contained in the span of the rows $(\cP_p)_{v_i}$, by choice of the
$v_i$ \eqref{eq.V}.  Therefore, it suffices to show the statement for
$v=v_i$. Writing $V_i$ for the $i$-th row of $V$ and $e_i$ for the
$i$-th canonical basis vector, we get
\begin{multline}
  (p(v_iw_1),...,p(v_iw_d))T_{\epsilon}y=V_iV^{-1}\begin{pmatrix}p(v_1)\\ \vdots\\ p(v_d)\end{pmatrix}
  =e_i'\begin{pmatrix}p(v_1)\\ \vdots\\ p(v_d)\end{pmatrix}=p(v_i).
\end{multline}
For the step $|w|\to |w|+1$, let $\tilde{w}=aw$ with $a\in\S$.  By
arguments which are analogous to those for $|w|=0$, it suffices to
consider $v=v_i$ referring to one of the row space generators
$(\cP_p)_{v_i}$ (while the induction hypothesis already holds for {\em
  all} $v,|v|\le\lceil\frac{n}{2}\rceil$)
\begin{equation}
  \begin{split}
    (&p(v_iw_1),...,p(v_iw_d))T_{\tilde{w}}y = V_iT_aT_wy = V_iV^{-1}W_aT_wy\\
    &= e_i'W_aT_wy = (p(v_iaw_1),...,p(v_iaw_d))T_wy \stackrel{(*)}{=} p(v_iaw) = p(v_i\tilde{w})
  \end{split}
\end{equation}
where $(*)$ is the induction hypothesis with $v=v_ia$, which applies
because of
$|v_ia|\le d\le\lceil\frac{n}{2}\rceil$.\qed\\

{\em Proof of Lemma~\ref{l.generators} cont.} Let $u\in\S^*$ such
that $|u|\le n$. Split $u=vw$ into two strings $v,w$ such that
$|v|\le\lceil\frac{n}{2}\rceil, |w|\le\lfloor\frac{n}{2}\rfloor$. We
compute
\begin{equation}
  \begin{split}
    x'T_uy &= x'T_vT_wy \stackrel{L.\ref{sl.generators.1}}{=} 
    = (p(vw_1),...,p(vw_d))T_wy \stackrel{L.~\ref{sl.generators.2}}{=} p(vw) = p(u).
  \end{split}
\end{equation} 
This yields (\ref{eq.oomrepr1}). For (\ref{eq.oomrepr2}) 
we compute
\begin{equation}
  \begin{split}
    (&p(v_iw_1),...,p(v_iw_d))\sum_{a\in\S}T_ay = \sum_{a\in\S}(p(v_iw_1),...,p(v_iw_d))T_ay \\
    &\stackrel{(L.\ref{sl.generators.2})}{=} \sum_{a\in\S}p(v_ia) = p(v_i) \stackrel{(L.\ref{sl.generators.2})}{=} (p(v_iw_1),...,p(v_iw_d))y
  \end{split}
\end{equation}
which yields the claim since $\spann\{(p(v_iw_1),...,p(v_iw_d))\mid
i=1,...,d\} = \C^d$.\qed\\

The step from the set-theoretic Lemma~\ref{l.generators} to the
our ideal-theoretic Theorem~\ref{t.generators} now follows from
standard arguments, as e.g.~listed in \cite{Cox}. In the following,
$\overline{A}$ denotes the Zariski closure of a set $A$, which is the
smallest affine algebraic variety which contains the set $A$, see
\cite{Cox}, sec.~4.4, def.~2.\\

In the following, we use
\begin{equation*}
F_d:=\im\mathbf{f}_{\cM_d,n}
\end{equation*}
as a simpler notation for the image of $\mathbf{f}_{\cM_d,n}$.\\

{\em Proof of Theorem~\ref{t.generators}}: We first compute
\begin{equation}
\begin{split}
V_{\cM_d,n} = \overline{F_d} 
&= \overline{(F_d\setminus F_{d-1}) \;\cup\; (F_{d-1}\setminus F_{d-2}) \;\cup\; ... \;\cup\; (F_1\setminus F_0) \;\cup\; F_0}\\
&= \overline{F_d\setminus F_{d-1}} \;\cup\; \overline{F_{d-1}\setminus F_{d-2}} \;\cup\; ... \;\cup\; \overline{F_1\setminus F_0} \;\cup\; \overline{F_0}
\end{split}
\end{equation}
where the last equation is an obvious consequence of the definition of
the Zariski closure: the Zariski closure agrees with the topological
closure if the latter one already is a variety.  The irreducibility of
$V_{\cM_d,n}$ implies that $V_{\cM_d,n}$ agrees with one of the
components $\overline{F_0},\overline{F_1\setminus
  F_0},...,\overline{F_d\setminus F_{d-1}}$. By
Theorem~\ref{t.mmmdim}, $\dim V_{\cM_d,n}=(|\S|-1)d^2+d$. Because of
$\dim\overline{F_e\setminus F_{e-1}} \le \overline{F_e} \le
\dim e^2(|\S|-1) + e$, which also follows from Theorem~\ref{t.mmmdim}, we
conclude that
\begin{equation}
V_{\cM_d,n} = \overline{F_{d}\setminus F_{d-1}}.
\end{equation}
By (\ref{eq.pd1d1}), it follows that (\ref{eq.generators})
is equivalent to 
\begin{equation}
\rk \cP_{p,\lfloor\frac{n}{2}\rfloor, \lceil\frac{n}{2}\rceil },\;
  \rk\cP_{p,\lceil\frac{n}{2}\rceil, \lfloor\frac{n}{2}\rfloor } \le d\quad
  \text{and}\quad
  \rk\cP_{p,d-1,d-1} \ge d.
\end{equation}
Application of Lemma~\ref{l.generators} reveals that
\begin{equation}
  \label{eq.Fdequiv}
  F_d\setminus F_{d-1} = A_{d+1,n}\setminus B_d
\end{equation}
where
\begin{eqnarray*}
  A_{d+1,n} & := & \{(p(v))_{v\in\S^n}\mid  \det (p(u_iv_j))_{1\le i,j\le d+1} = 0,\\ && \forall\; 0\le
  |u_i|,|v_j|\le\lceil\frac{n}{2}\rceil,|u_iv_j|\le n\}\\
  B_d & := & \{(p(v))_{v\in\S^n}\mid\det (p(u_iv_j))_{1\le i,j\le d} = 0,\\ &&
  \forall\; 0\le|u_i|,|v_j| \le d-1\}
\end{eqnarray*}
since $A_{d+1,n}$ consists of all $p$ such that all $(d+1)$-minors in
$\cP_{p,\lfloor\frac{n}{2}\rfloor, \lceil\frac{n}{2}\rceil }$ and
$\cP_{p,\lceil\frac{n}{2}\rceil, \lfloor\frac{n}{2}\rfloor }$ are zero
whereas $B_d$ encompasses all $p$ such that not all $d$-minors in
$\cP_{p,d-1,d-1}$ are zero.

As zero sets of determinantal (hence polynomial) equations, both
$A_{d+1,n}$ and $B_d$ are varieties, and recalling the definition
\eqref{eq.idealdef1},\eqref{eq.idealdef2} of $I_{d+1,n}$ and $J_d$, we
can conclude that these are just the ideals associated with
$A_{d+1,n}$ and $B_d$. By Hilbert's Nullstellensatz (see \cite[p.~174,
  theorem~6]{Cox}):
\begin{equation}
  I(A_{d+1,n}) = \rad I_{d+1,n}\quad\text{and}\quad I(B_d) = \rad J_d.
\end{equation}
The claim of Theorem~\ref{t.generators} now follows from the
interrelationship between quotients of ideals and differences of
varieties, as explicitly expressed by plugging $\rad I_{d+1,n}$ and
$J_d$ into $I$ and $J$ of the second statement of \cite[p.~192,
th.~7]{Cox} (the algebraically closed $k$ there becomes $\C$ here).\qed\\

\section{Algorithm}
\label{sec.algorithm}

Let $\S:=\{a,b\}$ be a binary-valued alphabet. The following algorithm
determines whether a probability distribution $\Prob:\S^n\to[0,1]$ is
due to a HMP on at most $d^*\le\frac{n+1}{2}$ hidden states, as supported
by

\begin{Thm}
  \label{t.implementation}
  \begin{sloppypar}Algorithm~\ref{a.finiteprob} below correctly
    decides and infers a HMP parametrization with at most $d$ hidden
    states for all but a lower-dimensional subvariety in
    $\cH_{d,+}$. \end{sloppypar}
\end{Thm}

That is, Theorem~\ref{t.implementation} establishes a {\em generic}
solution for Problem~\ref{pro.finite}. 

\begin{Alg}
\label{a.finiteprob}
\end{Alg}

{\sc IdentifyHMP($\Prob=(p(v))_{v\in\S^n}$)}

\begin{algorithmic}[1]
\STATE $e\leftarrow 1$
\WHILE{$e\le d:=\lfloor\frac{n+1}{2}\rfloor$}
\IF{$\rk\cP_{p,e-1,e-1} = \rk \cP_{p,\lfloor\frac{n}{2}\rfloor, \lceil\frac{n}{2}\rceil } = 
\rk\cP_{p,\lceil\frac{n}{2}\rceil, \lfloor\frac{n}{2}\rfloor } = e$}
\STATE $T_a,T_b,x\leftarrow$ {\sc InferFinitaryParam($\Prob,e$)}
\IF{$\det[T_a+T_b] > 0 \;\; \text{and} \;\; T_a[T_a+T_b]^{-1}$ is
  diagonalizable\\ such that all eigenvalues are different} 
\STATE $M,O_a,O_b,\pi\leftarrow$ {\sc InferHMMParam($T_a,T_b,x$)}
\IF{$(M,O_a,O_b,\pi)$ is stochastic} 
\STATE {\bf print} 'HMP on $e$ hidden states'
\STATE {\bf return} $M,O_a,O_b,\pi$ as parametrization
\ELSE
\STATE {\bf print} 'No HMP on $d$ hidden states'
\STATE {\bf return}
\ENDIF
\ENDIF
\ENDIF
\STATE $e\leftarrow e+1$
\ENDWHILE
\STATE {\bf print 'No HMP on $d$ hidden states'}
\end{algorithmic}

\medskip
\begin{sloppypar} {\sc InferFinitaryParam($\Prob,e$)} is a routine
  that computes an $e$-dimensional parametrization
  $(T_a,T_b,x)\in\cM_e$ for a finitary process. It works by computing
  $T_a,T_b$ and $x$ according to
  (\ref{eq.V},\ref{eq.x},\ref{eq.y},\ref{eq.Ta}) and subsequent
  application of Theorem \ref{t.oom} (note that any invertible $S$
  with $S^{-1}y=\mathbf{1}$ applies). According to Lemma
  \ref{l.generators} this works if
  \begin{equation*}
    \rk\cP_{p,e-1,e-1} = \rk \cP_{p,\lfloor\frac{n}{2}\rfloor, \lceil\frac{n}{2}\rceil } = 
    \rk\cP_{p,\lceil\frac{n}{2}\rceil, \lfloor\frac{n}{2}\rfloor } = e
  \end{equation*}
  which is guaranteed by step $3$.
\end{sloppypar}
\medskip

{\sc InferHMMParam($T_a,T_b,x$)} works if $[T_a+T_b]$ is invertible
and $T_a[T_a+T_b]^{-1}$ is diagonalizable such that all eigenvalues
$\lambda_1,...,\lambda_e$ are different, by Lemma
\ref{l.genericfiber}. In this case, one chooses (note that this is
possible!) $S\in\C^{e\times e}$ such that $S\mathbf{1}=\mathbf{1}$ and
\begin{equation*}
S^{-1}T_a[T_a+T_b]^{-1}S = \diag(\lambda_1,...,\lambda_e).
\end{equation*}
One then computes 
\begin{eqnarray*}
M &=& S^{-1}[T_a+T_b]S,\\
O_a,O_b&= &T_aM^{-1}, T_bM^{-1}\\ 
\pi & = & S'x.
\end{eqnarray*} 

\medskip The proof of Theorem~\ref{t.implementation} is based on
the following lemma for which we recall the definition of $\cN_d$,
see (\ref{eq.nullset}).

\begin{Lem}
\label{l.implementation}
Algorithm~\ref{a.finiteprob} can decide incorrectly only if
\begin{equation*}
\Prob\in\mathbf{f}_{\cH_e,n}(\cN_e)
\end{equation*}
for some $e=1,...,d$ where it may mistakenly output 'No
HMP on at most $d$ hidden states'.
\end{Lem}

Using Lemma~\ref{l.implementation} a proof of Theorem~\ref{t.implementation}
is easy:

{\em Proof of Theorem~\ref{t.implementation}.}
By Lemma~\ref{l.genericfiber}, $\cN_d$ forms a lower-dimensional
variety in $\cH_d$ and further, by Corollary~\ref{c.genericfiber},
$\cN_d\cap\cH_{d,+}$ also forms a lower-dimensional semialgebraic
set in $\cH_{d,+}$.\qed\\

{\em Proof of Lemma~\ref{l.implementation}.}
Let $\Theta\in\cH_{d,+}$ and
\begin{equation*}
  \Prob=\mathbf{f}_{\cH_d,n}(\Theta)
\end{equation*}
such that $\Prob$ is incorrectly classified as 'No HMP' by
Algorithm~\ref{a.finiteprob}.  We have to show that
\begin{equation*}
  \Theta\in\cN_e\quad\text{for some } e=1,...,d.
\end{equation*}
We recall the fundamental relationship (see
Propositions~\ref{p.finitaryobservations},\ref{p.hmmobservations})
\begin{equation}
  \label{eq.fundamental}
  \mathbf{f}_{\cH_e,n}(\cH_{e,+})\subset\im\mathbf{f}_{\cH_e,n}
  \subset\im\mathbf{f}_{\cM_e,n}
\end{equation}
In relation to (\ref{eq.fundamental}), Algorithm~\ref{a.finiteprob}
tests for membership from right to left in the $e$-th iteration of
the while loop, thereby stepwise approving or rejecting that
$\Prob\in\mathbf{f}_{\cH_e,n}(\cH_{e,+})\,[\subset\mathbf{f}_{\cH_d,n}(\cH_{d,+})]$.
First, by Lemma~\ref{l.generators}, {\em step} $3$ tests for
\begin{equation}
  \label{eq.dnotd-1}
  \Prob\in(\im\mathbf{f}_{\cM_e,n}\setminus\im\mathbf{f}_{\cM_{e-1},n}).
\end{equation}
Note that the case $\Prob\in\im\mathbf{f}_{\cM_{e-1},n}$ was excluded
in the iteration before.  This allows to infer an $e$-dimensional
parametrization for the respective finitary process in {\em step} $6$
(see the description of {\sc InferFinitaryParam} above). The {\bf if}
condition in {\em step} $7$ finally is the critical point; it
determines whether
\begin{equation}
  \label{eq.probinhmm}
  \Prob\in\mathbf{f}_{\cH_e,n}(\cH_{e,+}\setminus\cN_e)
\end{equation}
see the description of {\sc InferHMMParam}. If not, the algorithm
issues the output 'No HMP' which can be mistakenly due to either
$\Prob\in\mathbf{f}_{\cH_e,n}(\cH_{e,+}\cap\cN_e)\subset\mathbf{f}_{\cH_e,n}(\cN_e)$
or correctly due to either
$\Prob\in\mathbf{f}_{\cH_e,n}(\cN_e\setminus\cH_{e,+})$ or
$\Prob\in\im\mathbf{f}_{\cM_e,n}\setminus\im\mathbf{f}_{\cH_e,n}$.\par
By Lemma~\ref{l.genericfiber}, the parameters inferred in step $7$
are unique, up to permutations of rows and columns.  Therefore,
steps $8$ and $11$ decide correctly.\qed\\

\section{Acknowledgements}

First and foremost, the author would like to thank Bernd Sturmfels who
made the author aware of the algebraic challenges in this line of
research and has been a great source of inspiration throughout the
author's postdoctoral stay at UC Berkeley.  Special thanks to Jon
Hauenstein for Bertini computations and to Anton Leykin for advice on
both Macaulay and Bertini computations. Preliminary versions of that
work were presented on a workshop on Algebraic Statistics at the MSRI,
Berkeley, December 2008 and the annual AMS meeting in Lexington,
Kentucky, March 2010. The author would like to thank the organizers
and participants of those meetings for support and stimulating
discussions. The author also would like to thank David DesJardins for
a private donation that made the postdoctoral stay at UC Berkeley
possible.

\end{document}